\documentclass[12pt,reqno]{amsart}

\usepackage{amssymb}

\newtheorem{theorem}{Theorem}[section]
\newtheorem{proposition}[theorem]{Proposition}
\newtheorem{lemma}[theorem]{Lemma}

\numberwithin{equation}{section}

\begin{document}

\baselineskip=15pt

\title[Isomonodromic deformations of logarithmic connections and
stability]{Isomonodromic deformations of logarithmic connections and
stability}

\author[I. Biswas]{Indranil Biswas}

\address{School of Mathematics, Tata Institute of Fundamental
Research, Homi Bhabha Road, Bombay 400005, India}

\email{indranil@math.tifr.res.in}

\author[V. Heu]{Viktoria Heu}

\address{Institut de Recherche Math\'ematique Avanc\'ee, 7 rue 
Ren\'e-Descartes, 67084 Strasbourg Cedex, France}

\email{heu@math.unistra.fr}

\author[J. Hurtubise]{Jacques Hurtubise}

\address{Department of Mathematics, McGill University, Burnside
Hall, 805 Sherbrooke St. W., Montreal, Que. H3A 0B9, Canada}

\email{jacques.hurtubise@mcgill.ca}



\date{}

\begin{abstract}
Let $X_0$ be a compact connected Riemann surface of genus $g$ with
$D_0\, \subset\, X_0$ an ordered subset of cardinality $n$, and let
$E_G$ be a holomorphic principal $G$--bundle on $X_0$, where
$G$ is a reductive affine algebraic group defined over
$\mathbb C$, that is equipped with a
logarithmic connection $\nabla_0$ with polar divisor $D_0$. Let
$(\mathcal{E}_G\, , \nabla)$ be the universal isomonodromic deformation
of $(E_G\, ,\nabla_0)$ over the universal Teichm\"uller curve
$(\mathcal{X}, \mathcal{D})\,\longrightarrow\, \mathrm{Teich}_{g,n}$,
where $\mathrm{Teich}_{g,n}$ is the Teichm\"uller space for genus $g$ Riemann
surfaces with $n$--marked points.
We prove the following (see Section \ref{SecMain}):
\begin{enumerate}
\item Assume that $g\, \geq\, 2$ and $n\,=\, 0$. Then there is a
closed complex analytic subset $\mathcal{Y}\, \subset\,
\mathrm{Teich}_{g,n}$, of codimension at least $g$,
such that for any $t\,\in\,
\mathrm{Teich}_{g,n} \setminus \mathcal{Y}$, the principal $G$--bundle
$\mathcal{E}_G\vert_{{\mathcal X}_t}$ is semistable, where ${\mathcal X}_t$
is the compact Riemann surface over $t$.

\item Assume that $g\,\geq\, 1$, and if $g\,=\, 1$, then $n\, >\, 0$. Also, assume
that the monodromy representation for $\nabla_0$ does not factor through some proper
parabolic subgroup of $G$. Then there is a
closed complex analytic subset $\mathcal{Y}'\, \subset\,
\mathrm{Teich}_{g,n}$, of codimension at least $g$,
such that for any $t\,\in\,
\mathrm{Teich}_{g,n} \setminus \mathcal{Y}'$, the principal $G$--bundle
$\mathcal{E}_G\vert_{{\mathcal X}_t}$ is semistable.

\item Assume that $g\,\geq\, 2$. Assume
that the monodromy representation for $\nabla_0$ does not factor through some proper
parabolic subgroup of $G$. Then there is a
closed complex analytic subset $\mathcal{Y}''\, \subset\,
\mathrm{Teich}_{g,n}$, of codimension at least $g-1$,
such that for any $t\,\in\, \mathrm{Teich}_{g,n} \setminus \mathcal{Y}'$, the
principal $G$--bundle $\mathcal{E}_G\vert_{{\mathcal X}_t}$ is stable.
\end{enumerate}

In \cite{Viktoria1}, the second--named author proved the above results for
$G\,=\, \text{GL}(2,{\mathbb C})$.
\end{abstract}

\maketitle

\section{Introduction}

Take any quadruple of the form $(E \longrightarrow X\, , D\, ,\nabla)$, where $E$ is a
holomorphic vector bundle over a smooth connected complex variety $X$, and
$\nabla$ is an integrable logarithmic connection on $E$ singular over a simple normal
crossing divisor $D\, \subset\, X$. The monodromy functor associates to it
a representation $\rho_\nabla \,:\, \pi_1(X\setminus D,\, x_0)\,\longrightarrow\, 
\mathrm{GL}(E_{x_0})$, where $x_0\, \in\, X\setminus D$.
Altering the connection by a holomorphic automorphism $A$ of $E$
leads to a representation conjugated by $A(x_0)$. The monodromy functor produces 
an equivalence between the category of 
logarithmic connections $(E\, ,\nabla)$ on $(X,D)$ such that the real parts of
the residues lie in $[0\, ,1)$ and the category of equivalence classes
of linear representations of $\pi_1(X\setminus D,\, x_0)$. 
Given a monodromy representation $\rho$, one can consider the set of all 
logarithmic connections $(E\longrightarrow X\, , D\, ,\nabla)$
(with no condition on the residues) up to holomorphic
isomorphisms that produce the same monodromy 
representation $\rho\,=\,\rho_\nabla$ up to conjugation. All these connections are 
conjugated to each other by meromorphic gauge transformations with possible
poles over $D$ (see for example \cite{Sabbah}).

The classical Riemann--Hilbert problem can be formulated as follows:

\emph{Given a
representation $\rho \,:\, \pi_1(\mathbb{P}^1_{\mathbb C}\setminus D_0,\, x)
\,\longrightarrow\,
\mathrm{GL}(r,\mathbb{C})$, is there a logarithmic connection $(E\longrightarrow
\mathbb{P}^1_{\mathbb C}\,
, D_0\, ,\nabla)$ such that $\rho\,=\,\rho_\nabla$ and $E$ is holomorphically trivial?} 

The answer to this problem is 
\begin{enumerate}
\item positive if rank $r=2$ \cite{Plemelj}, \cite{Dekkers},
\item negative in general ($r\geq 3$) \cite{Bolibruch1},
\item positive if $\rho$ is irreducible \cite{Bolibruch2}, \cite{Kostov}. 
\end{enumerate}

On the other hand, the fundamental group $ \pi_1(\mathbb{P}^1_{\mathbb C}\setminus D_0,
\, x)$ depends only on the topological and not the complex structure of $\mathbb{P}^1_{
\mathbb C}\setminus D_0$. So given an integrable connection on $\mathbb{P}^1_{\mathbb C}
\setminus D_0$, one can consider variations of the
complex structure without changing the monodromy representation.
More precisely, consider the Teichm\"uller space $\mathrm{Teich}_{0,n}$ of the
$n$--pointed Riemann sphere together with the corresponding universal Teichm\"uller curve
$$
\tau\,:\,(\mathcal{X}\,=\,\mathbb{P}^1_{\mathbb C}\times \mathrm{Teich}_{0,n}\, ,
\mathcal{D})\,\longrightarrow\, \mathrm{Teich}_{0,n}\, ,
$$
where $n\,=\,{\rm degree}(D_0)$. Since $\mathrm{Teich}_{0,n}$ is contractible, the inclusion
$$
(\mathbb{P}^1_{\mathbb C}\, , D_t)\,:=\,\tau^{-1}(t)\,\hookrightarrow\, (\mathcal{X}
\, ,\mathcal{D})\, ,\ \ t\,\in\, \mathrm{Teich}_{0,n}\, ,
$$
induces an isomorphism $\pi_1(\mathbb{P}^1_{\mathbb C}\setminus D_t, \, x_t)
\,\simeq\, \pi_1(\mathcal{X}\setminus \mathcal{D}, \, x_t)$.
Hence by the Riemann--Hilbert correspondence, we can associate to
any logarithmic connection $(E_0\, ,\nabla_0)$ on $\mathbb{P}^1_{\mathbb C}$, with
polar divisor $D_0$, its \emph{universal isomonodromic deformation}: a flat logarithmic
connection $(\mathcal{E}\, , \nabla)$ over $\mathcal{X}$ with polar divisor
$\mathcal{D}$ that extends the connection $(E_0,\nabla_0)$. The conjugacy
class of the monodromy representation
for $\nabla\vert_{\tau^{-1}(t)}$ does not change as $t$ moves over
$\mathrm{Teich}_{0,n}$ (see for example \cite{Viktoria2}).

We are led to another Riemann--Hilbert problem:

\emph{Given a logarithmic connection 
$(E_0\, ,\nabla_0)$ on $\mathbb{P}^1_{\mathbb C}$ with polar divisor $D_0$ of degree
$n$, is there a point $t\,\in\, \mathrm{Teich}_{0,n}$ such that
the holomorphic vector bundle $E_t\,=\,\mathcal{E}\vert_{\mathbb{P}^1_{\mathbb C}
\times\{t\}}$ underlying the 
universal isomonodromic deformation $(\mathcal{E}\, , \nabla)$ is trivial?}

A partial answer to this question is given by the following theorem
of Bolibruch:

\begin{theorem}[{\cite{Bolibruch3}}]\label{thmBolibruch}
Let $(E_0\, ,\nabla_0)$ be an irreducible trace--free logarithmic rank two connection
with $n\,\geq\, 4$ poles on $\mathbb{P}^1_{\mathbb C}$ such that each singularity is
resonant. There is a proper closed complex analytic subset $\mathcal{Y}\,\subset\,
\mathrm{Teich}_{0,n}$ such that for all $t\,\in\, \mathrm{Teich}_{0,n}\setminus
\mathcal{Y}$, the holomorphic vector bundle $E_t\,=\,\mathcal{E}\vert_{\mathbb{P}^1_{
\mathbb C}\times\{t\}}$ underlying the universal isomonodromic deformation
$(\mathcal{E}\, , \nabla)$ of $(E_0,\nabla_0)$ is trivial. 
\end{theorem}

It should be mentioned that the condition in Theorem \ref{thmBolibruch}, that
each singularity is resonant, can actually be removed \cite{Viktoria1}.

{}From the Birkhoff--Grothendieck classification of holomorphic vector bundles on
$\mathbb{P}^1_{\mathbb C}$ it follows immediately
that the only semistable holomorphic vector
bundle of degree zero and rank $r$ on $\mathbb{P}^1_{\mathbb C}$ is the trivial bundle
$\mathcal{O}_{\mathbb{P}^1_{\mathbb C}}^{\oplus r}$. This leads to the following
more general question:
 
\emph{Given a representation $\rho \,:\, \pi_1(X\setminus D,\, x)
\,\longrightarrow\, \mathrm{GL}_r(\mathbb{C})$, where $X$ is a compact connected
Riemann surface, is there a logarithmic connection $(E\longrightarrow X\,, D\, ,\nabla)$
such that $\rho\,=\, \rho_\nabla$ and $E$ is semistable of degree zero?}

The answer to this problem is
\begin{enumerate}
\item negative in general \cite{Helene1},
\item positive if $\rho$ is irreducible \cite{Helene2}. 
\end{enumerate}

Let $\tau\,:\,(\mathcal{X}\, , \mathcal{D})\,\longrightarrow\, \mathrm{Teich}_{g,n}$ 
be the universal Teichm\"uller curve. In view of the above, it is natural to ask the 
following:

\emph{Given a logarithmic connection $(E_0\, ,
\nabla_0)$, with polar divisor $D_0$ of degree $n$ on a compact connected Riemann
surface $X_0$ of genus $g$, is there an element $t\,\in\, \mathrm{Teich}_{g,n}$
such that the holomorphic vector bundle $E_t\,=\,\mathcal{E}\vert_{{\mathcal X}_t}\,
\longrightarrow\, {\mathcal X}_t\,=\, \tau^{-1}(t)$ underlying the universal
isomonodromic deformation $(\mathcal{E}\, , \nabla)$ of
$(E_0\, , \nabla_0)$ is semistable?}

Note that we necessarily have ${\rm degree}(E_t)\,=\,{\rm degree}(E_0)$. Again,
Theorem \ref{thmBolibruch} can be generalized as follows. 

\begin{theorem}[{\cite{Viktoria1}}]\label{thmViktoria}
Let $(E_0\, ,\nabla_0)$ be an irreducible logarithmic rank two connection with polar
divisor $D_0$ of degree $n$ on a compact connected Riemann surface $X_0$
of genus $g$ such that $3g-3+n\, >\, 0$. Consider
its universal isomonodromic deformation $(\mathcal{E}\, , \nabla)$ over $\tau\,:\,
(\mathcal{X}\, , \mathcal{D})\,\longrightarrow\, \mathrm{Teich}_{g,n}$. There is a
closed complex analytic subset $\mathcal{Y}\, \subset\, \mathrm{Teich}_{g,n}$
(respectively, $\mathcal{Y}_s\, \subset\, \mathrm{Teich}_{g,n}$)
of codimension at least $g$ (respectively, $g-1$) such that for any $t\,\in\, 
\mathrm{Teich}_{g,n}\setminus\mathcal{Y}$, the vector bundle $E_t\,=\,
\mathcal{E}\vert_{{\mathcal X}_t}$, where $({\mathcal X}_t\, ,D_t)\,=\,
\tau^{-1}(t)$, is semistable (respectively, stable). 
\end{theorem}

Our aim here is to prove an analog of Theorem \ref{thmViktoria} in the
more general context of logarithmic connections on
principal $G$--bundles over a compact connected Riemann surface (see \cite{Philip} for
logarithmic connections on principal $G$--bundles).

Let $X_0$ be a compact connected Riemann surface of genus $g$, and let $D_0\,
\subset\, X_0$ be an ordered subset of it of cardinality $n$. Let $G$ be a
reductive affine algebraic group defined over
$\mathbb C$. Let $E_G$ be a holomorphic principal $G$--bundle on $X_0$ and
$\nabla_0$ a logarithmic connection on $E_G$ with polar divisor $D_0$. Let
$(\mathcal{E}_G\, , \nabla)$ be the universal isomonodromic deformation
of $(E_G\, ,\nabla_0)$ over the universal Teichm\"uller curve
$\tau\,:\, (\mathcal{X}\, , \mathcal{D})\,\longrightarrow\, \mathrm{Teich}_{g,n}$.
For any point $t\, \in\, \mathrm{Teich}_{g,n}$, the restriction
$\mathcal{E}_G\vert_{\tau^{-1}(t)}\,\longrightarrow\, {\mathcal X}_t\,:=\,
\tau^{-1}(t)$ will be denoted by $\mathcal{E}^t_G$.

We prove the following (see Section \ref{SecMain}):

\begin{theorem}\label{Result}
\mbox{}
\begin{enumerate}
\item Assume that $g\, \geq\, 2$ and $n\,=\, 0$. Then there is a
closed complex analytic subset $\mathcal{Y}\, \subset\,
\mathrm{Teich}_{g,n}$ of codimension at least $g$ such that for any $t\,\in\,
\mathrm{Teich}_{g,n} \setminus \mathcal{Y}$, the holomorphic principal $G$--bundle
$\mathcal{E}^t_G\,\longrightarrow\, {\mathcal X}_t$ is semistable.

\item Assume that $g\,\geq\, 1$, and if $g\,=\, 1$, then $n\, >\, 0$. Also, assume
that the monodromy representation for $\nabla_0$ does not factor through some proper
parabolic subgroup of $G$. Then there is a
closed complex analytic subset $\mathcal{Y}'\, \subset\,
\mathrm{Teich}_{g,n}$ of codimension at least $g$ such that for any $t\,\in\,
\mathrm{Teich}_{g,n} \setminus \mathcal{Y}'$, the holomorphic principal $G$--bundle
$\mathcal{E}^t_G$ is semistable.

\item Assume that $g\,\geq\, 2$. Assume that the monodromy representation for
$\nabla_0$ does not factor through some proper parabolic subgroup of $G$. Then there
is a closed complex analytic subset $\mathcal{Y}''\, \subset\,\mathrm{Teich}_{g,n}$ of
codimension at least $g-1$ such that for any $t\,\in\, \mathrm{Teich}_{g,n}\setminus
\mathcal{Y}'$, the holomorphic principal $G$--bundle $\mathcal{E}^t_G$ is stable.
\end{enumerate}
\end{theorem}

It is known that if a holomorphic principal $G$--bundle $E_G$ over a complex elliptic
curve admits a holomorphic connection, then $E_G$ is semistable. Therefore, a stronger
version of Theorem \ref{Result}(1) is valid when $g\,=\,1$.

\section{Infinitesimal deformations}

We first recall some classical results in deformation theory, and in the process 
setting up our notation.

\subsection{Deformations of a $n$--pointed curve}\label{se2.1}

Let $X_0$ be an irreducible smooth complex projective curve of genus $g$,
with $g\, >\, 0$, and let
$$
D_0\,:=\, \{x_1\, , \cdots\, , x_n\}\, \subset\, X_0
$$
be an ordered subset of cardinality $n$ (it may be zero). We assume that $n\, >\, 0$ if
$g\,=\,1$. This condition implies that the pair $(X_0\, ,D_0)$ does not have
any infinitesimal
automorphism, equivalently, the automorphism group of $(X_0\, ,D_0)$
is finite.

Let $B\,:=\, \text{Spec}({\mathbb C}[\epsilon]/\epsilon^2)$ be the spectrum of
the local ring. 
An \emph{infinitesimal deformation of $(X_0\, ,D_0)$} is given by a quadruple
\begin{equation}\label{e1}
({\mathcal X}\, ,q \, , {\mathcal D}\, ,f)\, ,
\end{equation}
where
\begin{itemize}
\item $q\, :\, {\mathcal X}\, \longrightarrow\, B$
is a smooth proper holomorphic morphism of relative dimension one,

\item $\mathcal{D}\,=\, (\mathcal{D}_1\, ,\cdots\, , \mathcal{D}_n)$ is
a collection of $n$ ordered disjoint sections of $q$, and

\item $f\, :\, X_0\,\longrightarrow\, {\mathcal X}$
is a holomorphic morphism such that
$$
f(X_0)\, \subset\, {\mathcal X}_0 \,:=\, q^{-1}(0) \ \ \text{ with }\ \
f(x_i)\,=\, \mathcal{D}_i(0) ~ \ \ \forall \ \ 1\,\leq\,i\, \leq\, n\, ,
$$
and the morphism $X_0\,\stackrel{f}{\longrightarrow}\, {\mathcal X}_0$ is an isomorphism.
\end{itemize}
The divisor $\sum_{i=1}^n \mathcal{D}_i(B)$ on $\mathcal X$ will also
be denoted by $\mathcal{D}$. For a vector bundle $\mathcal{V}$ on $\mathcal X$, the
vector bundle $\mathcal{V}\otimes_{{\mathcal O}_{\mathcal X}} {\mathcal O}_{\mathcal X}(-
{\mathcal D})$ will be denoted by $\mathcal{V}(-\mathcal{D})$.

The differential of $f$
$$
\mathrm{d}f\, :\, \mathrm{T}X_0\, \longrightarrow\,
f^*\mathrm{T}{\mathcal X}
$$
produces a homomorphism
$$
\mathrm{T}X_0(-D_0)\,:=\,
\mathrm{T}X_0\otimes_{{\mathcal O}_{X_0}}{\mathcal O}_{X_0}(-D_0)
\,\longrightarrow\, f^*\mathrm{T}{\mathcal X}(-\mathcal{D})
$$
which will
also be denoted by $\mathrm{d}f$. Consider the following short exact sequence of
coherent sheaves on $X_0$:
\begin{equation}\label{ti}
0\, \longrightarrow\, \mathrm{T}X_0(-D_0)\, \stackrel{\mathrm{d}f}{\longrightarrow}\,
f^*\mathrm{T}{\mathcal X}(-\mathcal{D})
\, \stackrel{h}{\longrightarrow}\, {\mathcal O}_{X_0} \, \longrightarrow\, 0\, ;
\end{equation}
note that the normal bundle to ${\mathcal X}_{0}\, \subset\, \mathcal X$ is the
pullback of $\mathrm{T}_{0}B$ by $q\vert_{{\mathcal X}_{0}}$, and hence
this normal bundle is identified with ${\mathcal O}_{X_0}$. Consider the
connecting homomorphism
\begin{equation}\label{ch}
{\mathbb C} \,=\, \mathrm{H}^0(X_0,\, {\mathcal O}_{X_0})\,
\stackrel{\phi}{\longrightarrow}\,\mathrm{H}^1(X_0,\, \mathrm{T}X_0(-D_0))
\end{equation}
in the long exact sequence of cohomologies
associated to the short exact sequence in \eqref{ti}. Let $1_{X_0}$ denote the
constant function $1$ on $X_0$. The cohomology element
\begin{equation}\label{e2}
\phi(1_{X_0})\, \in\, \mathrm{H}^1(X_0,\, \mathrm{T}X_0(-D_0))\, ,
\end{equation}
where $\phi$ is the homomorphism in \eqref{ch},
is the \emph{Kodaira--Spencer infinitesimal deformation class} for the
family in \eqref{e1}. If this infinitesimal deformation class is zero, then the
family $(\mathcal{X}\, , \mathcal{D})\,\longrightarrow\, B$ is isomorphic to
the trivial family $(X_0\times B \, , D_0\times B)\,\longrightarrow\, B$. 

\subsection{Deformations of a curve together with a
principal bundle}\label{SecBundleDef}

Take $(X_0\, ,D_0)$ as before.
Let $G$ be a connected reductive affine algebraic group defined over $\mathbb C$.
The Lie algebra of $G$ will be denoted by $\mathfrak g$. Let
\begin{equation}\label{g1}
p\, :\, E_G\,\longrightarrow\, X_0
\end{equation}
be a holomorphic principal $G$--bundle on $X_0$. The infinitesimal deformations of the triple
\begin{equation}\label{tr}
(X_0\, ,D_0\, ,E_G)
\end{equation}
are guided by the Atiyah bundle $\text{At}(E_G)\,\longrightarrow\, X_0$, the construction of which we shall
briefly recall (see \cite{At} for a more detailed exposition). Consider the direct image
$p_*\mathrm{T}E_G\, \longrightarrow\, X_0$, where $\mathrm{T}E_G$ is the holomorphic
tangent bundle of $E_G$, and $p$ is the projection in \eqref{g1}. It is a quasicoherent
sheaf equipped with an action of $G$ given by the action of $G$ on $E_G$. The
invariant part
$$
\text{At}(E_G)\,:=\, (p_*\mathrm{T}E_G)^G \, \subset\, (p_*\mathrm{T}E_G)
$$
is a vector bundle on $X_0$ of rank $1+\dim G$ which is known as the 
Atiyah bundle of $E_G$. Consequently, we have $\text{At}(E_G)\,=\, (TE_G)/G$. Let
$$
\text{ad}(E_G)\, :=\, E_G\times^G {\mathfrak g}\,\longrightarrow\, X_0
$$
be the adjoint vector bundle associated to $E_G$ for the adjoint action of
$G$ on its Lie algebra $\mathfrak g$. So the fibers of $\text{ad}(E_G)$ are
Lie algebras isomorphic to $\mathfrak g$. Let
$$
\mathrm{d}p\, :\, \mathrm{T}E_G\, \longrightarrow\, p^*\mathrm{T}X_0
$$
be the differential of the map $p$ in \eqref{g1}. Being $G$--equivariant
it produces a homomorphism $\text{At}(E_G)\,
\longrightarrow\,\mathrm{T}X_0$ which will also be denoted by $\mathrm{d}p$.
Now, the action of $G$ on $E_G$ produces an isomorphism $E_G\times{\mathfrak g}
\,\longrightarrow\, \text{kernel}(\mathrm{d}p)$. Therefore, we have
$\text{kernel}(\mathrm{d}p)/G\,=\, \text{ad}(E_G)$. In other words, the
above isomorphism $E_G\times{\mathfrak g}
\,\longrightarrow\, \text{kernel}(\mathrm{d}p)$ descends to an isomorphism
$$
\text{ad}(E_G)\,\stackrel{\sim}{\longrightarrow}\, (p_*(\text{kernel}(\mathrm{d}p)))^G
$$
that preserves the Lie algebra structure on the fibers of the two vector bundles
(the Lie algebra structure on the fibers of $(p_*(\text{kernel}(\mathrm{d}p)))^G$
is given by the Lie bracket of $G$--invariant vertical vector fields). Therefore,
$\text{At}(E_G)$ fits in the following short exact sequence of vector bundles on $X_0$
\begin{equation}\label{at1}
0\, \longrightarrow\, \text{ad}(E_G)\, \longrightarrow\,\text{At}(E_G)\,
\stackrel{\mathrm{d}p}{\longrightarrow}\, \mathrm{T}X_0\, \longrightarrow\, 0\, ,
\end{equation}
which is known as the Atiyah exact sequence for $E_G$. The logarithmic Atiyah bundle $\text{At}(E_G\,,D_0)$ is defined by 
$$
\text{At}(E_G\,,D_0)\,:=\, (\mathrm{d}p)^{-1}(\mathrm{T}X_0(-D_0))
\,\subset\, \text{At}(E_G) \,.
$$
{}From \eqref{at1} we have the short exact sequence of vector bundles on $X_0$
\begin{equation}\label{e3}
0\, \longrightarrow\, \text{ad}(E_G)\, \longrightarrow\,
\text{At}(E_G\,,D_0)\,
\stackrel{\sigma}{\longrightarrow}\, \mathrm{T}X_0(-D_0)\,
\longrightarrow\, 0\, ,
\end{equation}
which is called the {\it logarithmic Atiyah exact sequence}.
The above homomorphism $\sigma$ is the restriction of $\mathrm{d}p$ to
$\text{At}(E_G\,,D_0)\,\subset\, \text{At}(E_G)$.

An \emph{infinitesimal deformation} of the triple
$(X_0\, , D_0\, , E_G)$ in \eqref{tr} is a $6$--tuple 
 \begin{equation}\label{RelativeBundle}
({\mathcal X}\, ,q \, , {\mathcal D}\, ,f\, , {\mathcal E}_G\, ,\psi)\, ,
\end{equation}
where
\begin{itemize}
\item $({\mathcal X}\, ,q\, , {\mathcal D}\, ,f)$ in an infinitesimal
deformation of the $n$--pointed curve $(X_0\, , D_0)$ as in (\ref{e1}),

\item ${\mathcal E}_G\, \longrightarrow\,{\mathcal X}$ is a holomorphic principal
$G$--bundle, and

\item $\psi$ is a holomorphic isomorphism
\begin{equation}\label{psi}
\psi\, :\, E_G\, \longrightarrow\, f^*{\mathcal E}_G
\end{equation}
of principal $G$--bundles.
\end{itemize}
The logarithmic Atiyah bundle
$$
\text{At}({\mathcal E}_G\, ,{\mathcal D})\,\longrightarrow\, \mathcal X
$$
for $({\mathcal E}_G\, ,{\mathcal D})$ is the inverse image, in $\text{At}
({\mathcal E}_G)$, of the subsheaf $\mathrm{T}{\mathcal X}(-{\mathcal D}))\,\subset\,
\mathrm{T}{\mathcal X}$. We have the following short exact sequence of sheaves on $X_0$:
\begin{equation}\label{f2}
0\, \longrightarrow\, \text{At}(E_G\, ,D_0)\, \longrightarrow\, f^*\text{At}(
{\mathcal E}_G\, ,{\mathcal D})\, \longrightarrow\, {\mathcal O}_{X_0} \,
\longrightarrow\, 0
\end{equation}
given by $\psi$ in \eqref{psi}. Let
$$
{\mathbb C} \,=\, \mathrm{H}^0(X_0,\, {\mathcal O}_{X_0})\,
\stackrel{\widetilde\phi}{\longrightarrow}\,\mathrm{H}^1(X_0,\, \text{At}(E_G\,,D_0))
$$
be the connecting homomorphism in the long exact sequence of cohomologies associated
to \eqref{f2}. The cohomology element
\begin{equation}\label{e4}
{\widetilde\phi}(1_{X_0})\, \in\, \mathrm{H}^1(X_0,\, \text{At}(E_G\, , D_0))
\end{equation}
is the \emph{cohomology class} of the infinitesimal deformation of the triple
$(X_0\, , D_0\, , E_G)$ given by \eqref{RelativeBundle}. Let
$$
\sigma_*\, :\, \mathrm{H}^1(X_0,\, \text{At}(E_G\,, D_0))\, \longrightarrow\,
\mathrm{H}^1(X_0,\, \mathrm{T}X_0(-D_0))
$$
be the homomorphism given by the projection $\sigma$ in \eqref{e3}. From the
commutativity of the diagram
$$
\begin{matrix}
0 & \longrightarrow & \text{At}(E_G\, ,D_0) & \longrightarrow & f^*\text{At}( 
{\mathcal E}_G\, ,{\mathcal D}) &\longrightarrow & {\mathcal O}_{X_0} 
&\longrightarrow & 0\\
&&\Big\downarrow &&\Big\downarrow && \Vert\\
0 & \longrightarrow & \mathrm{T}X_0(-D_0) & \longrightarrow &
f^*\mathrm{T}{\mathcal X}(-\mathcal{D})& \longrightarrow & {\mathcal O}_{X_0}
&\longrightarrow & 0
\end{matrix}
$$
where the top and bottom rows are as in 
\eqref{f2} and \eqref{ti} respectively, it follows that
$$ 
\sigma_*({\widetilde\phi}(1_{X_0}))\,=\, {\phi}(1_{X_0})\, ,
$$
where ${\widetilde\phi}(1_{X_0})$ and ${\phi}(1_{X_0})$ are
constructed in \eqref{e4} and \eqref{e2} respectively. We note that $\sigma_*$ 
is the forgetful map that sends an infinitesimal deformation of $(X_0\, , D_0\, , 
E_G)$ to the underlying infinitesimal deformation of $(X_0\, , D_0)$ forgetting the
principal $G$--bundle.

\section{Obstruction to extension of a reduction of structure 
group to $P$}\label{se2.2}

Given a holomorphic reduction of structure group of $E_G$ to a parabolic subgroup of
$G$, our aim in this section is to compute the obstruction for
this reduction to extend to a reduction of
an infinitesimal deformation ${\mathcal E}_G\,\longrightarrow\, \mathcal X$ as in 
\eqref{RelativeBundle} (see the paragraph after the proof of Lemma \ref{lem1}).

A parabolic subgroup of $G$ is a connected Zariski closed subgroup $P$ such that
$G/P$ is a complete variety. Fix a parabolic subgroup
$
P\, \subset\, G\, .
$
The Lie algebra of $P$ will be denoted by $\mathfrak p$. As before, $E_G$ is
a holomorphic principal $G$--bundle on $X_0$. Let
\begin{equation}\label{e5}
E_P\, \subset\, E_G
\end{equation}
be a holomorphic reduction of structure group of $E_G$ to the subgroup
$P\,\subset\, G$. Let
$$
\text{ad}(E_P)\, :=\, E_P\times^P{\mathfrak p}\, \longrightarrow\, X_0
$$
be the adjoint vector bundle associated to $E_P$ for the adjoint action of
$P$ on its Lie algebra $\mathfrak p$. The vector bundle over $X_0$ associated to
the principal $P$--bundle $E_P$ for the adjoint action of $P$ on
the quotient ${\mathfrak g}/
{\mathfrak p}$ will be denoted by $E_P({\mathfrak g}/{\mathfrak p})$. So,
$E_P({\mathfrak g}/{\mathfrak p})\,=\, \text{ad}(E_G)/\text{ad}(E_P)$. The
logarithmic Atiyah bundle
for $(E_P\, , D_0)$ will be denoted by $\text{At}(E_P, D_0)$. We have
$$
\text{ad}(E_P)\, \subset\, \text{ad}(E_G)\ \ \text{ and }\ \
\text{At}(E_P, D_0)\, \subset\,\text{At}(E_G\, , D_0)\, ;
$$
both the quotient bundles above are identified with $E_P({\mathfrak g}/{\mathfrak p})$.
In other words, we have the following commutative diagram of vector bundles on $X_0$
\begin{equation}\label{e6}
\begin{matrix}
&& 0 && 0 \\
&& \Big\downarrow && \Big\downarrow\\
0& \longrightarrow & \text{ad}(E_P) & \longrightarrow & \text{At}(E_P\, , D_0) &
\stackrel{\beta}{\longrightarrow} & \mathrm{T}X_0(-D_0) & \longrightarrow & 0\\
&&\Big\downarrow && ~ \Big\downarrow \xi && \Vert\\
0& \longrightarrow & \text{ad}(E_G) & \longrightarrow & \text{At}(E_G\, , D_0) &
\stackrel{\sigma}{\longrightarrow} & \mathrm{T}X_0(-D_0) & \longrightarrow & 0\\
&&\,\,\,\,\,\, \Big\downarrow \mu_1 && ~ \Big\downarrow \mu \\
&& E_P({\mathfrak g}/{\mathfrak p}) & = & E_P({\mathfrak g}/{\mathfrak p})\\
&& \Big\downarrow && \Big\downarrow\\
&& 0 && 0
\end{matrix}
\end{equation}
where $\sigma$ is the homomorphism in \eqref{e3}. Let
\begin{equation}\label{e7}
\widetilde{\xi}\, :\, \mathrm{H}^1(X_0,\, \text{At}(E_P\, , D_0))\, \longrightarrow\,
\mathrm{H}^1(X_0,\,\text{At}(E_G\, , D_0))
\end{equation}
be the homomorphism induced by the canonical injection $\xi$ in \eqref{e6}.

Take any $({\mathcal X}\, ,q \, , {\mathcal D}\, ,f\, , {\mathcal E}_G\, ,\psi)$
as in \eqref{RelativeBundle}. Assume that the reduction $E_P\, \subset\,
E_G$ in \eqref{e5} extends to a holomorphic reduction of structure group
$${\mathcal E}_P\,\, \subset\, \mathcal E_G$$
to $P\,\subset\, G$
on $\mathcal X$. Consider the short exact sequence on $X_0$
\begin{equation}\label{f1}
0\, \longrightarrow\, \text{At}(E_P\, , D_0)\, \longrightarrow\,
f^*\text{At}({\mathcal E}_P\, , \mathcal{D})\,
\longrightarrow\, {\mathcal O}_{X_0}\, \longrightarrow\, 0\, ,
\end{equation}
where $\text{At}({\mathcal E}_P\, , \mathcal{D})\,\longrightarrow
\,{\mathcal X}$ is the logarithmic Atiyah bundle 
associated to the principal $P$--bundle ${\mathcal E}_P$, and
$f$ is the map in \eqref{e1}. Let
\begin{equation}\label{theta}
\theta\, \in\, \mathrm{H}^1(X_0,\, \text{At}(E_P\, , D_0))
\end{equation}
be the image of the constant function $1_{X_0}$ by the homomorphism
$$\mathrm{H}^0(X_0,\, {\mathcal O}_{X_0})\,\longrightarrow\,
\mathrm{H}^1(X_0,\, \text{At}(E_P\, , D_0))$$
in the long exact sequence of cohomologies associated to \eqref{f1}.

\begin{lemma}\label{lem1}
The cohomology class $\theta$ in \eqref{theta} satisfies the equation
$$
\widetilde{\xi}(\theta)\,=\, \widetilde{\phi}(1_{X_0})\, ,
$$
where $\widetilde{\xi}$ and $\widetilde{\phi}(1_{X_0})$ are constructed in \eqref{e7}
and \eqref{e4} respectively.
\end{lemma}

\begin{proof}
Consider the commutative diagram of vector bundles on $X_0$
\begin{equation}\label{f5}
\begin{matrix}
0& \longrightarrow & \text{At}(E_P\, , D_0) & \longrightarrow & 
f^*\text{At}({\mathcal E}_P\, , \mathcal{D}) & 
\longrightarrow & {\mathcal O}_{X_0} & \longrightarrow & 0\\ 
&&~ \Big\downarrow\xi && \Big\downarrow && \Vert\\
0& \longrightarrow & \text{At}(E_G\, , D_0) & 
\longrightarrow & f^*\text{At}({\mathcal E}_G\, , \mathcal{D}) & {\longrightarrow} & 
{\mathcal O}_{X_0} & \longrightarrow & 0
\end{matrix}
\end{equation}
where the top and bottom rows are as in \eqref{f1} and \eqref{f2} respectively,
and $\xi$ is the homomorphism in \eqref{e6}; the above homomorphism
$f^*\text{At}({\mathcal E}_P\, , \mathcal{D})\,\longrightarrow\,
f^*\text{At}({\mathcal E}_G\, , \mathcal{D})$ is the pullback of the
natural homomorphism $\text{At}({\mathcal E}_P\, , \mathcal{D})\,\longrightarrow\,
\text{At}({\mathcal E}_G\, , \mathcal{D})$.
In view of \eqref{f5}, the lemma follows by comparing the constructions of
$\theta$ and $\widetilde{\phi}(1_{X_0})$.
\end{proof}

Consider the homomorphism $\mu_*\, :\,
\mathrm{H}^1(X_0,\,\text{At}(E_G\, , D_0))\,\longrightarrow\,
\mathrm{H}^1(X_0,\, E_P({\mathfrak g}/{\mathfrak p}))$ induced by the
homomorphism $\mu$ in \eqref{e6}. From Lemma \ref{lem1} we conclude that
$\mu_*(\widetilde{\phi}(1_{X_0}))\,=\, 0$; to see this consider the long exact
sequence of cohomologies associated to the right vertical exact sequence
in \eqref{e6}. Therefore, $\mu_*(\widetilde{\phi}(1_{X_0}))$ is an
obstruction for the reduction $E_P\, \subset\, E_G$ to extend to a reduction
of ${\mathcal E}_G$ to $P$.

\section{Logarithmic connections and the second fundamental form}

In this section we characterize those infinitesimal deformations of the 
principal bundle $E_G$ on the $n$--pointed curve that arise from the isomonodromic
deformations.

\subsection{Canonical extension of a logarithmic connection}\label{SecConnections}

As before, let $p\,:\, E_G\, \longrightarrow\, X_0$ be a holomorphic principal $G$--bundle.
A \textit{logarithmic connection} on $E_G$ with polar divisor $D_0$ is a holomorphic
splitting of the logarithmic Atiyah exact
sequence in \eqref{e3}. In other words, a logarithmic connection is a homomorphism
\begin{equation}\label{de}
\delta\, :\, \mathrm{T}X_0(-D_0)\, \longrightarrow\, \text{At}(E_G\, , D_0)
\end{equation}
such that $\sigma\circ\delta\,=\, \text{Id}_{\mathrm{T}X_0(-D_0)}$, where
$\sigma$ is the homomorphism in \eqref{e3}. Note that given such a $\delta$, there
is a unique homomorphism
\begin{equation}\label{depp}
\delta''\, :\, \text{At}(E_G\, , D_0)\, \longrightarrow\,
\text{ad}(E_G)
\end{equation}
such that $\delta''\circ\delta\,=\, 0$, and the composition
$$
\text{ad}(E_G)\, \hookrightarrow\,\text{At}(E_G\, , D_0)\,
\stackrel{\delta''}{\longrightarrow}\, \text{ad}(E_G)
$$
(see \eqref{e3}) is the identity map of $\text{ad}(E_G)$.
As there are no nonzero $(2\, ,0)$--forms on $X_0$, all
logarithmic connections on $E_G$ are automatically integrable.

At the level of first order infinitesimal deformations, given a principal $G$--bundle
$$
{\mathcal E}_G\, \longrightarrow \,{\mathcal X}\, \stackrel{q}{\longrightarrow} \,B\, ,
$$
a logarithmic connection on ${\mathcal E}_G$ with polar divisor ${\mathcal D}$
is a homomorphism $\text{At}({\mathcal E}_G\, , {\mathcal D})\,\longrightarrow
\, \text{ad}({\mathcal E}_G)$ that splits the logarithmic Atiyah exact sequence
for ${\mathcal E}_G$. We note that a connection on ${\mathcal E}_G$
need not be integrable, as we have added an (infinitesimal) extra dimension.
 However, the Riemann--Hilbert correspondence for principal
$G$--bundles yields the following:

\begin{lemma}\label{LemRH}
Let $({\mathcal X}\, ,q \, , {\mathcal D}\, ,f)$
be an infinitesimal deformation of $(X_0\, ,D_0)$ as in \eqref{e1}. Let
$\delta$ be a logarithmic connection on a holomorphic principal $G$--bundle $E_G$
on $X_0$ with polar divisor $D_0$. Then there exists a unique pair
$(\mathcal{E}_G\, , \nabla)$, where
\begin{itemize}
\item $\mathcal{E}_G$ is a holomorphic principal $G$--bundle on ${\mathcal X}$, and

\item $\nabla$ is an integrable logarithmic connection on $\mathcal{E}_G$
with polar divisor ${\mathcal D}$,
\end{itemize}
such that $(f^*\mathcal{E}_G\, ,f^*\nabla)\,=\, (E_G\, , \delta)$.
\end{lemma}

Let us recall a few elements of the proof of this (classical) result.
Choose a covering 
$\mathfrak{U}$ of $X_0\setminus D_0$ by complex discs and a small neighborhood $U_i$ 
for each $x_i \,\in\, D_0$. Since $\delta$ is integrable, we can choose local 
charts for $E_G$ over $\mathfrak{U}$ such that all transition functions are 
constants. Now if the curve fits into an analytic family ${\mathcal X}\,\longrightarrow 
\,{\mathcal B}$, one can, restricting ${\mathcal B}$ if necessary, cover ${\mathcal X}$
by open subsets of the form $V_j\times B$, where $V_j$ are the open subsets of $X_0$
in the collection $\mathfrak{U}\cup \{U_i\}_{i=1}^n$. The isomonodromic deformation is 
then given by simply extending the transition maps by keeping them to be constant in 
deformation parameters.

The logarithmic connection $\delta$ gives a splitting of the 
logarithmic Atiyah bundle
$$ 
\text{At}(E_G\, , D_0)\, =\, \text{ad}(E_G) \oplus \mathrm{T}X_0(-D_0)\, .
$$
The corresponding cohomological decomposition
$$
\mathrm{H}^1(X_0, \text{At}(E_G\, , D_0) )\,=\,
\mathrm{H}^1(X_0,\text{ad}(E_G)) \oplus \mathrm{H}^1(X_0,\mathrm{T}X_0(-D_0))
$$
gives a splitting of the infinitesimal deformations of $(X_0\, ,D_0\, ,E_G)$ into
the infinitesimal deformations of $(X_0\, ,D_0)$ and the infinitesimal deformations
of $E_G$ (keeping $(X_0\, ,D_0)$ fixed). In other words, let
$$
\delta'\,:\, \mathrm{H}^1(X_0,\, \mathrm{T}X_0(-D_0))\, \longrightarrow\,
\mathrm{H}^1(X_0,\, \text{At}(E_G\, , D_0))
$$
be the homomorphism induced by the homomorphism
\begin{equation}\label{de12}
\delta\, :\, \mathrm{T}X_0(-D_0)
\, \longrightarrow\, \text{At}(E_G\, , D_0)
\end{equation}
in Lemma \ref{LemRH} defining the
logarithmic connection on $E_G$. Given an infinitesimal deformation $({\mathcal X}\, ,q
\, , {\mathcal D}\, ,f)$ of $(X_0\, ,D_0)$, the above homomorphism $\delta'$ produces
an infinitesimal deformation $({\mathcal X}\, ,q \, , {\mathcal D}\, ,f\, ,
{\mathcal E}_G\, ,\psi)$ of $(X_0\, ,D_0\, ,E_G)$. As explained above, this holomorphic
principal $G$--bundle ${\mathcal E}_G$ on $\mathcal X$ coincides with the holomorphic
principal $G$--bundle on $\mathcal X$ produced by the isomonodromic deformation in
Lemma \ref{LemRH}.

We will now construct the exact sequence in \eqref{f2} corresponding to the above
infinitesimal deformation $({\mathcal X}\, ,q \, , {\mathcal D}\, ,f\, , {\mathcal E}_G
\, ,\psi)$. Consider the injective homomorphism
$$
\mathrm{T}X_0(-D_0)\, \longrightarrow\, \text{At}(E_G\, , D_0)\oplus
f^*\mathrm{T}\mathcal{X}(-\mathcal{D})\, , ~\ ~ v\, \longmapsto\,
(\delta(v)\, , -(\mathrm{d}f)(v))\, ,
$$
where $\mathrm{d}f$ is the differential in \eqref{ti} and $\delta$
is the homomorphism in \eqref{de12}. The corresponding cokernel
$$
\text{At}^\delta (E_G\, , D_0)\, :=\, (\text{At}(E_G\, , D_0)\oplus
f^*\mathrm{T}\mathcal{X}(-\mathcal{D}))/(\mathrm{T}X_0(-D_0))
$$
possesses a canonical projection
\begin{equation}\label{what}
\widehat{\delta}\, :\, \text{At}^\delta (E_G\, , D_0)\, \longrightarrow\,{\mathcal O}_{X_0}\, , ~\ ~
(v\, ,w)\, \longmapsto\, h(w)\, ,
\end{equation}
where $h$ is the homomorphism in \eqref{ti}; note that the above homomorphism
$\widehat{\delta}$ is
well--defined because $h$ vanishes on the image of $\mathrm{T}X_0(-D_0)$ in $\text{At}(E_G\, , D_0)\oplus
f^*\mathrm{T}\mathcal{X}(-\mathcal{D})$. The kernel of $\widehat{\delta}$ is
identified with $\text{At}(E_G\, , D_0)$ by sending any $z\, \in\,
\text{At}(E_G\, , D_0)$ to the image in $\text{At}^\delta (E_G\, , D_0)$ of
$(z\, ,0)\,\in\, \text{At}(E_G\, , D_0)\oplus f^*\mathrm{T}\mathcal{X}(-\mathcal{D})$.
Therefore, we obtain the following exact sequence of vector bundles over $X_0$:
$$
0\, \longrightarrow\,{\rm At}(E_G\, , D_0)\, \longrightarrow\,{\rm At}^\delta (E_G\, , D_0) \,
\stackrel{\widehat{\delta}}{\longrightarrow}\, {\mathcal O}_{X_0} \,
\longrightarrow\, 0\, ,
$$
This exact sequence coincides with the one in \eqref{f2}. In particular, we have
${\rm At}^\delta (E_G\, , D_0)\,=\,f^*\text{At}({\mathcal E}_G)$.

Consider the projection
$$
\text{At}(E_G\, , D_0)\oplus
f^*\mathrm{T}\mathcal{X}(-\mathcal{D})\, \longrightarrow\, \text{ad}(E_G)\, ,
\ \ (z_1\, ,z_2)\, \longmapsto\, \delta''(z_1)\, ,
$$
where $\delta''$ is constructed in \eqref{depp} from $\delta$. It vanishes
on the image of $\mathrm{T}X_0(-D_0)$, yielding a projection
\begin{equation}\label{ar}
\lambda\, :\, \text{At}^\delta (E_G\, , D_0)\,\longrightarrow\, \text{ad}(E_G)\, .
\end{equation}
Let $\nabla''\, :\, \text{At}({\mathcal E}_G\, , {\mathcal D})\,
\longrightarrow\, \text{ad}({\mathcal E}_G)$ be the homomorphism 
given by the logarithmic connection $\nabla$ in Lemma \ref{LemRH}.
The homomorphism in \eqref{ar} fits in the commutative diagram
\begin{equation}\label{comm}
\begin{matrix}
\text{At}^\delta (E_G\, , D_0)& \stackrel{\lambda}{\longrightarrow} & \text{ad}(E_G)\\
\Vert && \Vert\\
f^*\text{At} ({\mathcal E}_G\, , {\mathcal D})
&\stackrel{f^*\nabla''}{\longrightarrow} & f^*\text{ad}({\mathcal E}_G)
\end{matrix}
\end{equation}
(the vertical identifications are evident).

We summarize the above constructions in the following lemma:

\begin{lemma}\label{lem2}
Given $({\mathcal X}\, ,q\, , {\mathcal D}\, ,f)$ as in \eqref{e1}, and also
given a logarithmic connection
$\delta$ on a holomorphic principal $G$--bundle $E_G\, \longrightarrow\, X_0$, the
exact sequence in \eqref{f2} corresponding to the infinitesimal deformation
of $(X_0\,,D_0\,, E_G)$ in Lemma \ref{LemRH} is
$$
0\, \longrightarrow\,{\rm At}(E_G\, , D_0)\, \longrightarrow\,{\rm At}^\delta
(E_G\, , D_0) \,\stackrel{\widehat{\delta}}{\longrightarrow}\, {\mathcal O}_{X_0}
\, \longrightarrow\, 0\, ,
$$
where $\widehat{\delta}$ is constructed in \eqref{what}.
\end{lemma}

\subsection{The second fundamental form}\label{SecFundForm}

Fix a logarithmic connection $\delta$ on $(E_G\,,D_0)$ as in \eqref{de}. Take a
holomorphic reduction of structure group $E_P\,\subset\, E_G$ to $P$ as in
\eqref{e5}. The composition
\begin{equation}\label{sff}
S(\delta)\,:=\,
\mu\circ\delta\, :\, \mathrm{T}X_0(-D_0)\,\longrightarrow\, E_P({\mathfrak g}/{\mathfrak p})\, ,
\end{equation}
where $\mu$ is constructed in \eqref{e6}, is called the \textit{second fundamental
form} of $E_P$ for the connection $\delta$. We note that $\delta$ is induced by
a logarithmic connection on the holomorphic principal $P$--bundle $E_P$ if and only if
we have $S(\delta)\,=\, 0$.

Assume that $E_P$ satisfies the condition that
$
S(\delta)\,\not=\, 0\, .
$
Let
\begin{equation}\label{l}
{\mathcal L}\, \subset\, E_P({\mathfrak g}/{\mathfrak p})
\end{equation}
be the holomorphic line subbundle generated by the image of the homomorphism
$S(\delta)$ in \eqref{sff}. More
precisely, ${\mathcal L}$ is the inverse image, in $E_P({\mathfrak g}/{\mathfrak p})$, of
the torsion part of the quotient $E_P({\mathfrak g}/{\mathfrak p})/(S(\delta)(\mathrm{T}X_0(-D_0)))$.
Now consider the homomorphism
\begin{equation}\label{sdp}
S(\delta)' \,:\, \mathrm{T}X_0(-D_0)\,\longrightarrow\,
{\mathcal L}
\end{equation}
given by the second fundamental form $S(\delta)$. Let
\begin{equation}\label{sp}
S'\, :\, \mathrm{H}^1(X_0,\, \mathrm{T}X_0(-D_0))\,\longrightarrow\,
\mathrm{H}^1(X_0,\, {\mathcal L})
\end{equation}
be the homomorphism of cohomologies induced by $S(\delta)'$ in \eqref{sdp}.

\begin{proposition}\label{prop1}
As before, let $\delta$ be a logarithmic connection on $E_G\,\longrightarrow\, X_0$ with
polar divisor $D_0$, and let $({\mathcal X}\, ,q \, , {\mathcal D}\, ,f)$ be an
infinitesimal deformation of $(X_0,D_0)$. Let ${\mathcal E}_G\,\longrightarrow\,{\mathcal
X}$ be the isomonodromic deformation of $\delta$ along $(\mathcal{X}\, ,
\mathcal{D})$ obtained in Lemma \ref{LemRH}. Let $E_P\, \subset\, E_G$ be a holomorphic
reduction of structure group to $P$ over $X_0$ that extends to a holomorphic
reduction of structure group ${\mathcal E}_P\, \subset\, {\mathcal E}_G$ to $P$ over
${\mathcal X}$. Then
$$
S'(\phi(1_{X_0}))\,=\, 0\, ,
$$
where $\phi(1_{X_0})$ is the cohomology class constructed in \eqref{e2}
corresponding to $({\mathcal X}\, ,q \, , {\mathcal D}\, ,f)$,
and $S'$ is constructed in \eqref{sp}.
\end{proposition}

\begin{proof}
Consider the inverse images
$$
\text{At}_P(E_G\, , D_0)\,:=\, \mu^{-1}({\mathcal L})\,\subset\, \text{At}(E_G\, , D_0)
\ \ \text{ and} \ \ \
\text{ad}_P(E_G )\,:=\, \mu^{-1}_1({\mathcal L})\,\subset\, \text{ad}(E_G)\, ,
$$
where $\mu$ and $\mu_1$ are the quotient maps in \eqref{e6}, and
$\mathcal L$ is constructed in \eqref{l}. These two vector bundles fit in the
following commutative diagram produced from \eqref{e6}:
\begin{equation}\label{f6}
\begin{matrix}
&& 0 && 0 \\
&& \Big\downarrow && \Big\downarrow\\
0& \longrightarrow & \text{ad}(E_P) & \longrightarrow & \text{At}(E_P\, , D_0) &
\stackrel{\beta}{\longrightarrow} & \mathrm{T}X_0(-D_0) & \longrightarrow & 0\\
&&\Big\downarrow && ~ \Big\downarrow \xi && \Vert\\
0& \longrightarrow & \text{ad}_P(E_G) & \longrightarrow & \text{At}_P(E_G\, , D_0) &
\stackrel{\gamma}{\longrightarrow} & \mathrm{T}X_0(-D_0) & \longrightarrow & 0\\
&& \,\,\,\,\, \Big\downarrow \mu_1 && ~ \Big\downarrow \mu \\
&& {\mathcal L} & = & {\mathcal L}\\
&& \Big\downarrow && \Big\downarrow\\
&& 0 && 0
\end{matrix}
\end{equation}
By the construction of $\text{At}_P(E_G\, , D_0)$, the
connection homomorphism $\delta$ in \eqref{de12} factors through a homomorphism
$$
\delta^1\, :\,\mathrm{T}X_0(-D_0)\, \longrightarrow\, \text{At}_P(E_G\, , D_0)\, .
$$
Consider the homomorphism
\begin{equation}\label{de1s}
\delta^1_*\, :\, \mathrm{H}^1(X_0,\, \mathrm{T}X_0(-D_0))\, \longrightarrow\,
\mathrm{H}^1(X_0,\,\text{At}_P(E_G\, , D_0))
\end{equation}
induced by the above homomorphism $\delta^1$, and let
\begin{equation}\label{tau}
\Phi\,:=\, \delta^1_*(\phi(1_{X_0}))\, \in\, \mathrm{H}^1(X_0,\,\text{At}_P(E_G\, , D_0))
\end{equation}
be the image of the cohomology class $\phi(1_{X_0})$ that characterizes the
deformation $(\mathcal{X}\,,\mathcal{D})$ as in \eqref{e2}.

As in the statement of the proposition, let ${\mathcal E}_P\, \longrightarrow\,
{\mathcal X}$ be a holomorphic extension of the
reduction $E_P$. Note that from \eqref{f5}, \eqref{e6} and Lemma \ref{lem2}
we have
$$
\text{At}(E_G\, , D_0)/\text{At}(E_P\, , D_0)\,=\,
E_P({\mathfrak g}/{\mathfrak p})\,=\, (f^*\text{At}({\mathcal E}_G\, ,
{\mathcal D}))/ (f^*\text{At}({\mathcal E}_P\, ,{\mathcal D}))
$$
$$
\,=\, {\rm At}^\delta(E_G\, , D_0)/
f^*\text{At}({\mathcal E}_P\, , \mathcal{D})\, .
$$
Let $\mu_2\, :\, {\rm At}^\delta (E_G\, , D_0)\,\longrightarrow\,
E_P({\mathfrak g}/{\mathfrak p})$ be the above quotient map. Define
$$
{\rm At}^\delta_P (E_G\, , D_0)\, :=\, \mu^{-1}_2({\mathcal L})
\, \subset\, {\rm At}^\delta (E_G\, , D_0)\, ,
$$
where $\mathcal L$ is constructed in \eqref{l}.

Now we have the following commutative diagram:
\begin{equation}\label{dia}
\begin{matrix}
0 & \longrightarrow & \mathrm{T}X_0(-D_0) & \longrightarrow &
f^*\mathrm{T}{\mathcal X}(-\mathcal{D})& \longrightarrow & {\mathcal O}_{X_0}
&\longrightarrow & 0\\
&&~\Big\downarrow\delta &&\,\,\,\,\,\,\,\,\, \Big\downarrow f^*\nabla && \Vert\\
0& \longrightarrow & \text{At}_P(E_G\, , D_0) &
\longrightarrow & {\rm At}^\delta_P (E_G\, , D_0) & {\longrightarrow} &
{\mathcal O}_{X_0} & \longrightarrow & 0
\end{matrix}
\end{equation}
where the bottom exact sequence is obtained from \eqref{f5}, and the
top exact sequence is as in \eqref{ti} (see also \eqref{comm}). Let
\begin{equation}\label{nu}
\nu\, :\, \mathrm{H}^0(X_0, \, {\mathcal O}_{X_0})\, \longrightarrow\,
\mathrm{H}^1(X_0,\, \text{At}_P(E_G\, , D_0))
\end{equation}
be the connecting homomorphism in the long exact sequence of cohomologies
associated to the bottom exact sequence in \eqref{dia}. From the commutativity
of \eqref{dia} and the construction of $\phi(1_{X_0})$
(see \eqref{e2}) it follows that
\begin{equation}\label{phex}
\nu(1_{X_0})\,=\, \delta^1_*(\phi(1_{X_0}))\,=\, \Phi\, ,
\end{equation}
where $\nu$ and $\delta^1_*$ are the homomorphisms constructed in \eqref{nu}
and \eqref{de1s} respectively, and $\Phi$ is the cohomology class in \eqref{tau}.

The diagram in \eqref{f5} produces the diagram
\begin{equation}\label{g5}
\begin{matrix}
&& 0 && 0 \\
&& \Big\downarrow && \Big\downarrow\\
0& \longrightarrow & \text{At}(E_P\, , D_0) & \longrightarrow & 
f^*\text{At}({\mathcal E}_P\, , D_0) &
\longrightarrow & {\mathcal O}_{X_0} & \longrightarrow & 0\\
&&~ \Big\downarrow\xi && \Big\downarrow && \Vert\\
0& \longrightarrow & \text{At}_P(E_G\, , D_0) &
\longrightarrow & {\rm At}^\delta_P (E_G\, , D_0) & {\longrightarrow} &
{\mathcal O}_{X_0} & \longrightarrow & 0\\
&&~ \Big\downarrow\mu && ~ \Big\downarrow\\
&& {\mathcal L} & = & {\mathcal L}\\
&& \Big\downarrow && \Big\downarrow\\
&& 0 && 0
\end{matrix}
\end{equation}
where $\xi$ and $\mu$ are the homomorphisms in \eqref{f6}. Using this
diagram we can check that
\begin{equation}\label{id}
\Phi\,=\, \xi_*(\theta)\, ,
\end{equation}
where $\theta$ is the cohomology classes in \eqref{theta}, and
$$
\xi_*\, :\, \mathrm{H}^1(X_0,\, \text{At}(E_P\, , D_0))\, \longrightarrow\, 
\mathrm{H}^1(X_0,\,\text{At}_P(E_G\, , D_0))
$$
is the homomorphism induced by $\xi$ in \eqref{g5}. Indeed, \eqref{id}
follows from \eqref{phex}, the commutativity of \eqref{g5}
and the construction of $\theta$.

Let $\mu_*\, :\, \mathrm{H}^1(X_0,\, \text{At}_P(E_G\, , D_0))\,\longrightarrow\,
\mathrm{H}^1(X_0,\, {\mathcal L})$ be the homomorphism induced by the homomorphism $\mu$
in \eqref{g5}. It is straight--forward to check that
$$\mu_*(\Phi)\,=\, S'(\phi(1_{X_0}))$$ (see \eqref{sp}, \eqref{e2} and \eqref{tau}
for $S'$, $\phi(1_{X_0})$ and $\Phi$ respectively). Indeed, this follows from
\eqref{phex} and the definition of $S(\delta)$ in \eqref{sff}. Combining this
with \eqref{id}, we have
$$
S'(\phi(1_{X_0}))\,=\, \mu_*(\Phi)\,=\, \mu_*(\xi_*(\theta))\, .
$$
Since $\mu\circ\xi\,=\, 0$ (see \eqref{g5}), it now follows that
$S'(\phi(1_{X_0}))\,=\,0$.
\end{proof}

\section{Logarithmic connections and semistability of the underlying
principal bundle}\label{SecMain}

Let $\mathcal{T}_{g,n}$ be the Teichm\"uller space
of genus $g$ compact Riemann surfaces with $n$ ordered marked points. As before,
we assume that $g\, >\, 0$, and if $g\,=\,1$, then $n\, >\, 0$.
Let
$$
\tau\, :\, \mathcal{C} \,\longrightarrow\, \mathcal{T}_{g,n}
$$
be the universal Teichm\"uller curve with $n$ ordered sections $\Sigma$. The fiber
of $\mathcal{C}$ over any point $t\,\in\, \mathcal{T}_{g,n}$ will be denoted by
${\mathcal C}_t$. The ordered subset ${\mathcal C}_t\cap \Sigma\,\subset\,
{\mathcal C}_t$ will be denoted by $\Sigma_t$.

Take a $n$--pointed Riemann surface $(C_0\, ,\Sigma_0)$ of genus $g$,
which is represented by
a point of $\mathcal{T}_{g,n}$. Let
\begin{equation}\label{na0}
\nabla_0
\end{equation}
be a logarithmic connection on a holomorphic principal $G$--bundle $F_G\,\longrightarrow
\, C_0$ with polar divisor $\Sigma_0$. By the Riemann--Hilbert correspondence, the
connection $\nabla_0$ produces a flat 
(isomonodromic) logarithmic connection $\nabla$ on a holomorphic principal
$G$--bundle $\mathcal{F}_G\,\longrightarrow\, \mathcal{C}$ with polar
divisor $\Sigma$.

The following lemma is a special case of the main theorem in \cite{Nitsure} (see 
also \cite{Sh}). Although the families of $G$--bundles considered in \cite{Nitsure} 
are algebraic, all arguments there go through in the analytic case of our 
interest with obvious modifications.

\begin{lemma}[\cite{Nitsure}] \label{lemNitsure}
Let $\mathcal{F}_G\,\longrightarrow\, \mathcal{C}\,\longrightarrow\,\mathcal{T}_{g,n}$
be as above. For each Harder--Narasimhan type $\kappa$, the set
$$
\mathcal{Y}_\kappa \,:=\, \{t\,\in\, \mathcal{T}_{g,n} ~\mid ~
\mathcal{F}_G\vert_{{\mathcal C}_t}\  \text{ is\ of\ type }\ \kappa\}
$$
is a (possibly empty) locally closed complex analytic subspace of $\mathcal{T}_{g,n}$.
More precisely, for each Harder--Narasimhan type $\kappa$, the union 
$\mathcal{Y}_{\leq \kappa}\,:=\,\bigcup_{\kappa'\leq \kappa}\mathcal{Y}_{\kappa'}$ is
a closed complex analytic subset of $\mathcal{T}_{g,n}$.
Moreover, the principal $G$--bundle
$$
\mathcal{F}_G\vert_{\tau^{-1}(\mathcal{Y}_\kappa)}\,\longrightarrow\,
\tau^{-1}(\mathcal{Y}_\kappa)
$$
possesses a canonical holomorphic 
reduction of structure group inducing the Harder--Narasimhan
reduction of $\mathcal{F}_G\vert_{{\mathcal C}_t}$ for every $t\,\in \,\mathcal{Y}_\kappa$.
\end{lemma}

In the following two Sections \ref{se5.1} and \ref{se5.2}, we will see that under 
certain assumptions, the only Harder-Narasimhan stratum $\mathcal{Y}_\kappa$ of 
maximal dimension $\dim(\mathcal{T}_{g,n})\,=\,3g-3+n$ is the trivial one, in the sense 
that the Harder--Narasimhan parabolic subgroup is $G$ itself. In other words, if 
the principal $G$--bundle $F_G$ is not semistable, and therefore has a non-trivial 
Harder--Narasimhan reduction $E_P\, \subset\, E_G$ to a certain parabolic subgroup 
$P\, \subsetneq\, G$, then there is always an isomonodromic deformation in which 
direction the reduction $E_P$ is obstructed meaning it does not extend.

\subsection{The case of $n = 0$}\label{se5.1}

In this subsection we assume that $n\,=\,0$. So, we have $g\, >\, 1$.

\begin{theorem}\label{prop2}
There is a closed complex analytic subset ${\mathcal Y}\, \subset\,
{\mathcal T}_{g,0}$ of codimension at least $g$
such that for every $t\, \in\, {\mathcal T}_{g,0}\setminus {\mathcal Y}$, the
holomorphic principal $G$--bundle $\mathcal{F}_G\vert_{{\mathcal C}_t}\,
\longrightarrow\, {\mathcal C}_t$ is semistable.
\end{theorem}

\begin{proof}
Let ${\mathcal Y}\, \subset\, {\mathcal T}_{g,0}$ denote the (finite) union of all 
Harder-Narasimhan strata ${\mathcal Y}_\kappa$ as in Lemma \ref{lemNitsure} with 
nontrivial Harder-Narasimhan type $\kappa$. From Lemma \ref{lemNitsure} we know
that ${\mathcal Y}$ is a closed complex analytic subset of ${\mathcal T}_{g,0}$.

Take any $t\, \in\, \mathcal{Y}_\kappa\, \subset\, {\mathcal Y}$. Let $E_G\,=\,
\mathcal{F}_G\vert_{{\mathcal C}_t}$ be the holomorphic
principal $G$--bundle on $X_0\,:=\, {\mathcal C}_t$. The holomorphic
connection on $E_G$ obtained by restricting
$\nabla$ will be denoted by $\delta$. Since $E_G$ is not semistable,
there is a proper parabolic subgroup $P\, \subsetneq\, G$ and a holomorphic
reduction of structure group $E_P\, \subset\, E_G$ to $P$, such that $E_P$ is the
Harder--Narasimhan reduction \cite{Be}, \cite{AAB}; the type of this
Harder--Narasimhan reduction is $\kappa$. From Lemma \ref{lemNitsure} we
know that $E_P$ extends to a holomorphic reduction of structure group of the principal
$G$--bundle $\mathcal{F}_G\vert_{\tau^{-1}(\mathcal{Y}_\kappa)}$ to the subgroup $P$.

Let $\text{ad}(E_P)$ and $\text{ad}(E_G)$ be the adjoint vector bundles
of $E_P$ and $E_G$ respectively. Consider the vector bundle
$$
\text{ad}(E_G)/\text{ad}(E_P)\,=\, E_P({\mathfrak g}/{\mathfrak p})
$$
(see \eqref{e6}). We know that
\begin{equation}\label{deg2}
\mu_{\rm max}(E_P({\mathfrak g}/{\mathfrak p}))\, <\, 0
\end{equation}
\cite[p. 705]{AAB} (see sixth line from bottom). In particular
\begin{equation}\label{deg3}
\text{degree}(E_P({\mathfrak g}/{\mathfrak p})) \, <\, 0\, .
\end{equation}
A holomorphic connection on $E_G$ induces a holomorphic connection
on $\text{ad}(E_G)$, hence $\text{degree}(\text{ad}(E_G))\,=\, 0$.
Combining this with \eqref{deg3} it follows that
$\text{degree}(\text{ad}(E_P))\,>\, 0$, because
$\text{ad}(E_G)/\text{ad}(E_P)\,=\,E_P({\mathfrak g}/{\mathfrak p})$.
Since $\text{degree}(\text{ad}(E_P))\,\not=\, 0$, the holomorphic vector
bundle $\text{ad}(E_P)$ does not admit any holomorphic connection, hence
the principal $P$--bundle
$E_P$ does not admit a holomorphic connection. Consequently,
the second fundamental form
$S(\delta)$ of $E_P$ for $\delta$ (see \eqref{sff}) is nonzero.

Using the second fundamental form $S(\delta)$, construct the holomorphic
line subbundle
$$
{\mathcal L}\, \subset\, E_P({\mathfrak g}/{\mathfrak p})
$$
as done in \eqref{l}. From \eqref{deg2} we have
\begin{equation}\label{deg}
\text{degree}({\mathcal L}) \, <\, 0\, .
\end{equation}
Consider the short exact sequence of sheaves on $X_0$
\begin{equation}\label{f7}
0\,\longrightarrow\, \mathrm{T}X_0 \,
\stackrel{S(\delta)'}{\longrightarrow}\, {\mathcal L}
\,\longrightarrow\, T^\delta \,:=\, {\mathcal L}/(S(\delta)(\mathrm{T}X_0))
\,\longrightarrow\, 0\, ,
\end{equation}
where $S(\delta)'$ is constructed in \eqref{sdp}; note that
$T^\delta$ is a torsion sheaf because $S(\delta)'\,\not=\, 0$
(recall that $S(\delta)\,\not=\, 0$). From \eqref{deg} it follows
that
$$
\text{degree}(T^\delta)\,=\, \text{degree}({\mathcal L}) -
\text{degree}(\mathrm{T}X_0) \, <\, - \text{degree}(\mathrm{T}X_0)\, =\, 2g-2\, .
$$
So, we have
$$
\dim \mathrm{H}^0(X_0,\, T^\delta)\,=\, \text{degree}(T^\delta)\, <\, 2g-2
\, \,=\, \dim \mathrm{H}^1(X_0,\, \mathrm{T}X_0)+1-g\, .
$$
This implies that
\begin{equation}\label{li}
\dim \mathrm{H}^1(X_0,\, \mathrm{T}X_0)- \dim \mathrm{H}^0(X_0,\, T^\delta)
\,\geq\, g\, .
\end{equation}
Consider the long exact sequence of cohomologies
$$
\mathrm{H}^0(X_0,\, T^\delta)\,\longrightarrow\,\mathrm{H}^1(X_0,\,
\mathrm{T}X_0)\,\stackrel{\zeta}{\longrightarrow}\,\mathrm{H}^1(X_0,\, {\mathcal L})
$$
associated to the short exact sequence of sheaves in \eqref{f7}. From \eqref{li} it
follows that
\begin{equation}\label{codi}
\dim \zeta(\mathrm{H}^1(X_0,\,
\mathrm{T}X_0))\, \geq\, g\, .
\end{equation}

Since the reduction $E_P$ extends to a holomorphic reduction of structure group of the
principal $G$--bundle $\mathcal{F}_G\vert_{\tau^{-1}(\mathcal{Y}_\kappa)}$
to the subgroup $P$, combining \eqref{codi} and Proposition \ref{prop1}
we conclude that the codimension of the complex analytic subset
$\mathcal{Y}_\kappa\, \subset\, {\mathcal T}_{g,0}$ is at least $g$.
This completes the proof of the theorem.
\end{proof}

\subsection{When $n$ is arbitrary}\label{se5.2}

Now we assume that $n\,>\,0$.

A logarithmic connection $\eta$ on a holomorphic principal $G$--bundle 
$F_G\,\longrightarrow\, X_0$ is called \textit{reducible} if there is pair 
$(P\, ,F_P)$, where $P\, \subsetneq\, G$ is a parabolic subgroup and 
$F_P\, \subset\, F_G$ is a holomorphic reduction of structure group of $F_G$ to $P$,
such that $\eta$ is induced by a connection on $F_P$. Note that $\eta$ is 
induced by a connection on $F_P$ if and only if the second fundamental 
form of $F_P$ for $\eta$ vanishes identically. A connection is called 
\textit{irreducible} if it is not reducible or, equivalently, if the 
monodromy representation of the corresponding flat principal $G$--bundle 
does not factor through any proper parabolic subgroup of $G$.

\begin{proposition}\label{propo3}
Assume that the logarithmic connection $\nabla_0$ in \eqref{na0} is irreducible.
Then there is a closed complex analytic subset ${\mathcal Y}\, \subset\,
{\mathcal T}_{g,n}$ of codimension at least $g$
such that for every $t\, \in\, {\mathcal T}_{g,n}\setminus {\mathcal Y}$, the
holomorphic principal $G$--bundle $\mathcal{F}_G\vert_{{\mathcal C}_t}$
is semistable.
\end{proposition}

\begin{proof}
The proof of Theorem \ref{prop2} goes through after some obvious modifications.
As before, let ${\mathcal Y}\, \subset\, {\mathcal T}_{g,n}$ be the
locus of all points $t$ such that the principal $G$--bundle
$\mathcal{F}_G\vert_{{\mathcal C}_t}$ is not semistable.
Take any point $t\, \in\, {\mathcal Y}_\kappa\, \subset\,{\mathcal Y}$. Let $E_G\,=\,
\mathcal{F}_G\vert_{{\mathcal C}_t}$ be the holomorphic
principal $G$--bundle on $X_0\,:=\, {\mathcal C}_t$. The logarithmic
connection on $E_G$ with polar divisor $D_0\,:=\,\Sigma_t$ obtained by
restricting $\nabla$ will be denoted by $\delta$.

Let $E_P\, \subset\, E_G$ be the Harder--Narasimhan reduction; its
type is $\kappa$. Since $\nabla_0$ is 
irreducible, the second fundamental form $S(\delta)$ of $E_P$ for $\delta$ (see 
\eqref{sff}) is nonzero. We note that for the monodromy of a logarithmic connection, 
the property of being irreducible is preserved under isomonodromic deformations. 
Consider the short exact sequence of sheaves on $X_0$
\begin{equation}\label{l7}
0\,\longrightarrow\, \mathrm{T}X_0(-D_0) \,
\stackrel{S(\delta)'}{\longrightarrow}\, {\mathcal L}
\,\longrightarrow\, T^\delta \,:=\, {\mathcal L}/(S(\delta)(\mathrm{T}X_0(-D_0)))
\,\longrightarrow\, 0\, ,
\end{equation}
where $S(\delta)'$ is constructed in \eqref{sdp}. As before, $\text{degree}(
{\mathcal L})\, <\, 0$, because $\mu_{\rm max}(E_P({\mathfrak g}/{\mathfrak p}))
\, <\, 0$. So,
$$
\text{degree}(T^\delta)\,=\, \text{degree}({\mathcal L})- \text{degree}
(\mathrm{T}X_0(-D_0))\, <\, -\text{degree}(\mathrm{T}X_0(-D_0))\,=\,2g-2+n\, .
$$ We now have
$$
\dim \mathrm{H}^0(X_0,\, T^\delta)\,=\, \text{degree}(T^\delta)\, <\, 2g-2+n
\, \,=\, \dim \mathrm{H}^1(X_0,\, \mathrm{T}X_0(-D_0))+1-g\, .
$$
Hence the dimension of the image of the homomorphism
\begin{equation}\label{codi2}
\mathrm{H}^1(X_0,\, \mathrm{T}X_0(-D_0))\, \longrightarrow\,
\mathrm{H}^1(X_0,\, {\mathcal L})
\end{equation}
in the long exact sequence of cohomologies
associated to \eqref{l7} is at least $g$. Since the reduction $E_P$ extends to a
holomorphic reduction of $\mathcal{F}_G\vert_{\tau^{-1}(\mathcal{Y}_\kappa)}$ to
$P$, and the dimension of the image of the homomorphism in \eqref{codi2} is
at least $g$, from Proposition \ref{prop1} we conclude that the codimension of
$\mathcal{Y}_\kappa \, \subset\, {\mathcal T}_{g,0}$ is at least $g$.
\end{proof}

\subsection{Stability of underlying principal bundle}\label{se5.3}

We now assume that $g\, \geq\, 2$.

\begin{proposition}\label{thm1}
Assume that the logarithmic connection $\nabla_0$ in \eqref{na0} is irreducible.
There is a closed analytic subset ${\mathcal Y}'\, \subset\, {\mathcal T}$ of
codimension at least $g-1$
such that for every $t\, \in\, {\mathcal T}_g\setminus {\mathcal Y}$, the holomorphic
principal $G$--bundle $\mathcal{F}_G\vert_{{\mathcal C}_t}$ is stable.
\end{proposition}

\begin{proof}
The proof is identical to the proof of Proposition \ref{propo3}.
If $E_G\,=\, \mathcal{F}_G\vert_{{\mathcal C}_t}$ is not stable, there
is a maximal parabolic subgroup $P\,\subsetneq\, G$ and a holomorphic reduction
of structure group $E_P\, \subset\, E_G$ to $P$, such that the quotient bundle
$$
\text{ad}(E_G)/\text{ad}(E_P)\,=\, E_P({\mathfrak g}/{\mathfrak p})
$$
is semistable of degree zero. Therefore, we have $\text{degree}({\mathcal L})
\,\leq\, 0$. This implies that $$\text{degree}(T^\delta)
\,\leq\, 2g-2-n\, .$$ 
Hence the dimension of the image of the homomorphism
$\mathrm{H}^1(X_0,\, \mathrm{T}X_0(-D_0))\, \longrightarrow\,
\mathrm{H}^1(X_0,\, {\mathcal L})$ in the long exact sequence of cohomologies
associated to \eqref{l7} is at least $g-1$.
\end{proof}

\section*{Acknowledgements}

We thank the referee for helpful comments. We thank Universit\'e de Brest for 
hospitality where the work was initiated. The first author is supported by a J. C. 
Bose Fellowship. The second author is supported by ANR-13-BS01-0001-01 and 
ANR-13-JS01-0002-01.

%%%%%%%%%%%%%%%%%%%%%%%%%%%%%%%%%%%%%%%%%%%%%%%%%%%%%%%%%%%%%%

\end{document}